\documentclass[11pt
]{amsart}

\usepackage{hyperref}
\hypersetup{bookmarksdepth=3}

\usepackage{geometry}
\usepackage{amsmath,amssymb}
\usepackage{
mathrsfs}
\usepackage{esint}

\usepackage[usenames]{color}
\usepackage{colortbl}

    \normalbaselines                          %

\newtheorem{thm}{Theorem}[section]
\newtheorem{lm}[thm]{Lemma}

\theoremstyle{definition}

\theoremstyle{remark}

\numberwithin{equation}{section}

\newcommand{\f}{\varphi}

\newcommand{\la}{\lambda}

\newcommand{\al}{\alpha}

\newcommand{\cz}{Calder\'{o}n--Zygmund\ }

\newcommand{\B}{\mathbb{B}}
\newcommand{\bb}{B^{\epsilon_0}}

\newcommand{\e}{\varepsilon}

\newcommand{\s}{\sigma}

\newcommand{\Om}{\Omega}
\newcommand{\om}{\omega}

\newcommand{\E}{\mathbb{E}}
\newcommand{\F}{{\mathcal F}}

\newcommand{\R}{\mathbb{R}}

\newcommand{\FF}{\bf{F}}
\newcommand{\ff}{{\bf f}}
\newcommand{\g}{{\bf g}}
\newcommand{\WW}{\bf{w}}

\newcommand{\eps}{\epsilon}

\newcommand{\PP}{\mathcal{P}}

\newcommand{\wt}{\widetilde}
\newcommand{\La}{\langle}
\newcommand{\Ra}{\rangle}

\def\cyr{\fontencoding{OT2}\fontfamily{wncyr}\selectfont}
\DeclareTextFontCommand{\textcyr}{\cyr}

{\end{list}}

\newcounter{vremennyj}

\begin{document}

\title[$A_1$ conjecture]{A Bellman function counterexample to the $A_1$ conjecture:  the blow-up of the weak norm  estimates of weighted singular operators}
\author{Fedor Nazarov}
\author{Alexander Reznikov}
\author{Vasily Vasyunin}
\author{Alexander Volberg}
\date{}

\begin{abstract}
We consider several weak type estimates for  singular operators using the Bellman function approach. We disprove the $A_1$ conjecture, which is a weaker conjecture than Muckenhoupt--Wheeden conjecture  disproved by Reguera--Thiele.
\end{abstract}

\maketitle

\section{Introduction}
\label{intro}

Maria Reguera \cite{MR} disproved Muckenhoupt--Wheeden conjecture. Then Maria Reguera and Christoph Thiele disproved Muckenhoupt--Wheeden conjecture \cite{RT}, which required that the Hilbert transform would map $L^1(Mw)$ into $L^{1,\infty}(w)$.  It has been suggested in P\'erez' paper  \cite{P} that there should exist such a counterexample, also \cite{P} has several very interesting positive results, where $Mw$ is replaced by a slightly bigger maximal function, in particular by $M^2w$ (which is equivalent to a certain Orlicz maximal function). 

Here we strengthen Reguera and Reguera--Thiele results by disproving the so called $A_1$ conjecture (which also seems to be rather old and due to Muckenhoupt).
The reader can get acquainted with the best so far positive result on $A_1$ conjecture in the paper \cite{LOP}. 

The $A_1$ conjecture stated that the Hilbert transform would map $L^1(w)$ to $L^{1,\infty}(w)$ with norm bounded by constant times $[w]_{A_1}$ (the $A_1$ ``norm" of $w$). Recall that $[w]_{A_1}:=\sup\frac{Mw(x)}{w(x)}$. Therefore, $A_1$ conjecture is weaker than Muckenhoupt--Wheeden conjecture, and, hence, it is more difficult to disprove it.  And, in fact, in \cite{MR}, \cite{RT} the $A_1$ norm of the weight is uncontrolled, while we need to construct a rather ``smooth" $w$ to build our counterexample.

The $A_1$ conjecture is  also called a weak Muckenhoupt--Wheeden conjecture. We prove that the linear estimate in weak Muckenhoupt--Wheeden conjecture is impossible, and, moreover, the growth of the weak norm of the the martingale transform and the weak norm of the Hilbert transform from $L^1(w)$ into $L^{1,\infty}(w)$ is at least $c\, [w]_{A_1}\log^{\frac15 -\epsilon} [w]_{A_1}$. Paper \cite{LOP} gives an estimate from above for such a norm: it is $\le C\, [w]_{A_1}\log [w]_{A_1}$. We believe that this latter estimate might be sharp and that our estimates from below can be improved. 

The plan of the paper: first we repeat the result of \cite{RVaVo}, where the exact Bellman function for the unweighted weak estimate of the martingale transform has been constructed. Then we show the logarithmic blow-up for the weighted estimate of the martingale transform in the end-point case $w\in A_1$. Then we adapt this result to obtain the same speed of blow-up for the Hilbert transform.

\medskip


\section{Unweighted weak type of $0$ shift}
\label{unweighted}

Here we review the work \cite{RVaVo}, where the Bellman function and the extremizers were constructed for the unweighted martingale transform.
The unweighted problem is much easier than the weighted problem that we consider in the current article. However, a glance at a simpler problem helps us to set up a more difficult one and to understand the difficulties. So we start with unweighted martingale transform, and briefly recall the reader the set up and some of the results of \cite{RVaVo}.

\bigskip

We are on $I_0:=[0,1]$. As always $D$ denote the dyadic lattice. We consider the operator
$$
\f\rightarrow \sum_{I\subseteq I_0, I\in D} \eps_I (\f, h_I) h_I\,,
$$
where $-1\le \eps_I\le 1$. Notice that the sum does not contain the constant term.

Put
$$
F:=\langle |\varphi|\rangle_{I}\,,\,f:=\langle \varphi\rangle_{I}\,,
$$
and introduce  the following function:
$$
B(F,f, \la):= \sup\,\frac1{|I|}|\{x\in I: \sum_{J\subseteq I, J\in D} \eps_J(\f, h_J) h_J(x)>\la\}|\,,
$$
where the $sup$ is taken over all $-1\le\eps_J\le 1, J\in D,\, J\subseteq I$, and over all $\varphi\in L^1(I)$ such that $
F:=\langle |\varphi|\rangle_{I}\,,\,f:=\langle \varphi\rangle_{I}
$, $ h_I$ are normalized in $L^2(\R)$ Haar function of the  cube (interval) $I$, and $|\cdot |$ denote Lebesgue measure.
Recall that
$$
h_I(x):=\begin{cases} \frac1{\sqrt{|I|}}\,,\, x\in I_+\\ -\frac1{\sqrt{|I|}}\,,\, x\in I_-\end{cases}
$$

This function is defined in a convex domain $\Om\subset \R^3$: $\Om:=\{(F,f,\la)\in \R^3: |f| \le F\}$.

\medskip

\noindent {\bf Remark.} Function $B$ should not be indexed by $I$ because it does not depend on $I$. We will use this soon.

\subsection{The main inequality}
\label{MI}

\begin{thm}
\label{tuda}
Let $P, P_+,P_-\in \Om, P=(F,f,\la)$, $P_+=(F+\al, f+\beta, \la +\beta)$, $P_-=(F-\al, f-\beta, \la -\beta)$. Then
\begin{equation}
\label{mi1}
B(P)-\frac12(B(P_+)+B(P_-))\ge 0\,.
\end{equation}
At the same time, if $P, P_+,P_-\in \Om, P=(F,f,\la)$, $P_+=(F+\al, f+\beta, \la -\beta)$, $P_-=(F-\al, f-\beta, \la +\beta)$. Then
\begin{equation}
\label{mi2}
B(P)-\frac12(B(P_+)+B(P_-))\ge 0\,.
\end{equation}
\end{thm}

\begin{proof}
Fix $P, P_+,P_-\in \Om, P=(F,f,\la)$, $P_+=(F+\al, f+\beta, \la +\beta)$, $P_-=(F-\al, f-\beta, \la -\beta)$.
Let $\varphi_+, \varphi_-$ be functions giving the supremum in $B(P_+), B(P_-)$ respectively up to a small number $\eta>0$.
Using the remark above we think that $\varphi_+$ is on $I_+$ and $\varphi_-$ is on $I_-$. Consider
$$
\f(x):=\begin{cases} \varphi_+(x)\,,\, x\in I_+\\ \varphi_-(x)\,,\, x\in I_-\end{cases}
$$
Notice that then
\begin{equation}
\label{beta1}
(\f, h_I)\cdot\frac1{\sqrt{|I|}} = \beta\,.
\end{equation}
Then it is easy to see that
\begin{equation}
\label{av1}
\langle | \f| \rangle_I = F=P_1, \,\,\,\langle \f \rangle_I =f=P_2\,.
\end{equation}
Notice that for $x\in I_+$ using \eqref{beta1}, we get if $\eps_I =-1$
$$
\frac1{|I|}|\{x\in I_+: \sum_{J\subseteq I, J\in D} \eps_J(\f, h_J) h_J(x)>\la\}|= \frac1{|I|}|\{x\in I_+: \sum_{J\subseteq I_+, J\in D} \eps_J(\f, h_J) h_J(x)>\la+\beta\}|
$$
$$
=\frac1{2|I_+|}|\{x\in I_+: \sum_{J\subseteq I_+, J\in D} \eps_J(\varphi_+, h_J) h_J(x)>P_{+,3}\}|\ge \frac12 B(P_+)-\eta\,.
$$
Similarly, for $x\in I_-$ using \eqref{beta1}, we get if $\eps_I =-1$
$$
\frac1{|I|}|\{x\in I_-: \sum_{J\subseteq I, J\in D} \eps_J(\f, h_J) h_J(x)>\la\}|= \frac1{|I|}|\{x\in I_-: \sum_{J\subseteq I_+, J\in D} \eps_J(\f, h_J) h_J(x)>\la-\beta\}|
$$
$$
=\frac1{2|I_-|}|\{x\in I_-: \sum_{J\subseteq I_-, J\in D} \eps_J(\varphi_-, h_J) h_J(x)>P_{-,3}\}|\ge \frac12 B(P_-)-\eta\,.
$$
Combining the two left hand sides we obtain for $\eps_I=-1$
$$
\frac1{|I|}|\{x\in I_+: \sum_{J\subseteq I, J\in D} \eps_J(\f, h_J) h_J(x)>\la\}|\ge \frac 12 (B(P_+)+B(P_-))-2\eta\,.
$$
Let us use now the simple information \eqref{av1}: if we take the supremum in the left hand side over all functions $\varphi$, such that $\langle |\varphi | \rangle_I =F, \langle \varphi \rangle_I =f $, and supremum over all $\eps_J\in [-1,1]$ (only $\eps_I=-1$ stays fixed), we get a quantity smaller or equal than the one, where we have the supremum over all functions $\varphi$, such that $\langle |\varphi | \rangle =F, \langle \varphi \rangle_I =f $, and an unrestricted  supremum over all $\eps_J\in [- 1,1]$. The latter quantity is of course $B(F, f, \la)$. So we proved \eqref{mi1}.

To prove \eqref{mi2} we repeat verbatim the same reasoning, only keeping now $\eps_I=1$. We are done.

\end{proof}

Denote
$$
T\f:=\sum_{J\subseteq I, J\in D} \eps_J(\f, h_J) h_J(x)\,.
$$
It is a dyadic singular operator (actually, it is a family of operators enumerated by sequences of $\eps_I\in [-1, 1]$). To prove that it is of weak type is the same as to prove
\begin{equation}
\label{wt1}
B(F,f,\la)\le \frac{C\,F}{\la}, \,\la >0\,.
\end{equation}
Our $B$ satisfies \eqref{mi1}, \eqref{mi2}. We consider this as concavity conditions.

Let us make the change of variables,
$(F,f, \la)\rightarrow (F, y_1, y_2)$:
$$
y_1:=\frac12 (\la+ f)\,,\,\, y_2:=\frac12 (\la-f)\,.
$$
Denote
$$
M(F, y_1, y_2) := B(F, y_1-y_2, y_1+y_2)= B(F, f, \la).
$$

In terms of function $M$ Theorem \ref{tuda} reads as follows:
\begin{thm}
\label{tudaM}
The function $M$ is defined in the domain $G:=\{(F,y_1,y_2): |y_1 -y_2|\le F\}$,  and
for each fixed $y_2$, $M(F, y_1, \cdot)$ is concave and for each fixed $y_1$, $M(F, \cdot, y_2)$ is concave.
\end{thm}

Abusing the language we will call by the same letter $B$ (correspondingly, $M$) {\it any}  function satisfying \eqref{mi1}, \eqref{mi2} (correspondingly satisfying Theorem \ref{tudaM}).

\bigskip

It is not difficult to obtain one more condition, the so-called {\it obstacle condition}:
\begin{lm}
\label{obstLemma}
\begin{equation}
\label{obst}
\text{If}\,\, \la < F\,\,\text{then}\,\, B(F,f, \la)=1.
\end{equation}
\end{lm}

\begin{proof}
Let us first consider the case $f=F$, which can be viewed as the case of non-negative functions $\phi$.
Fix $\la_0$ and $\eps>0$, let $\varphi$ be a  non-negative function on $I=[0,1]$ such that it looks like $(\la_0 +\eps)\delta_{0}$, and $F=f:=\int_0^1 \varphi \, dx = \la_0 +\eps>\la_0$.  Namely, $\varphi$ is zero on the set of measure $1-\tau$, and an almost $\delta$ function times $\la_0+\eps$ on a small interval of measure $\tau$.

 As it looks as a multiple of delta function, it can be written down as $\la_0 +\eps +H$, where $H$ is a combination of Haar functions, and a martingale transform of $\phi$, namely,  $-H=\la_0+\eps>\la_0$ on a set of measure $1-\tau$ with an arbitrary small $\tau$ (the smallness is independent of $\la_0$ and $\eps$).  Then the example of $\varphi$ shows that
 $$
 B(\la_0+\eps, \la_0+\eps, \la_0)\ge 1-\tau\,.
 $$
 We have to consider the case of $f<F$ as well. If $f>\la_0$, the construction is the same. Namely, consider $\Phi:=\varphi +aS$, where $S$ is a Haar function with very small support in a small dyadic interval $\ell$ (say, of measure smaller than $\tau$) and normalized in $L^1$, let $\ell$ be contained in the set, where $\varphi$ is small ($\varphi$ is small essentially on almost the whole interval, because it looks like a positive multiple of the delta function), and ensure that $\int S\,dx=0$, and $\int |S|\,dx=1$.
 Then the example of $\varphi$ shows that
 $$
 B(\int_0^1 |\Phi|\,dx, \la_0+\eps, \la_0)\ge 1-2\tau\,.
 $$
 By varying $a$  from $0$ to $\infty$ we can reach $\int |\Phi|\,dx = F$ for any $F\ge \la_0+\eps$. Therefore, making first $\tau\rightarrow 0$ and then $\eps\rightarrow 0$, we prove \eqref{obst}.

We are left to consider the case $F>\la_0\ge f$. Choose $\eps$ and $\tau$ much smaller than, say, $\frac1{10}(F-f)$. Consider the same function $\varphi$, as above. Let $H$ be  the first Haar function, namely $H= -1$ on $I_-=[0, 1/2]$ and $H= 1$ on $I_+=[1/2,1]$. Let us consider now $\psi:= \varphi+c_1\cdot H-c_2$, $c_1>c_2>0$. Then
 $$
 \langle \psi\rangle = \la_0+\eps -c_2,\, \langle|\psi|\rangle = \la_0+ c_1 +O(\tau).
 $$
 It is easy now to choose $c_1, c_2$ such that the first average above is equal to a given number $f$, and the second one is equal to a given $F$, $F>f$. Now on the set $E$ of measure $1-\tau$ we have $\psi= c_1H -c_2$. On the other hand $\psi= \la_0+\eps +c_1 H- c_2 + H_1$, where $H_1$ is a combination of Haar function, each of which is orthogonal to $H$.
 
 Hence, $-H_1 = \la_0 +\eps >\la_0$ on $E$ of measure $1-\tau$. But $-H_1$ is the martingale transform of $\psi$ in our sense. In fact, we just consider the Haar decomposition of $\psi$, forget the constant term, and multiply all Haar coefficients on $-1$ except the first one, which is got multiplied by $0$.
 
 We obtain that $B(F, f, \la_0) \ge 1-\tau$. We are done.


\end{proof}

\bigskip

\begin{thm}
\label{obratno}
Let $B\ge 0$ satisfy \eqref{mi1}, \eqref{mi2}. (Equivalently, let the corresponding $M\ge 0$ be concave in $(F, y_1)$ and in $(F, y_2)$.) Let
$B$ satisfy \eqref{wt1}, or, equivalently,
\begin{equation}
\label{wt2}
M(F,y_1,y_2)\le \frac{C\,F}{y_1+y_2}, \,y_1 +y_2>0\,.
\end{equation}
Let $B(F, f, \la)=1$ if $\la<0$.
Then we have the weak type estimate with constant at most $C$ for all $T$ uniformly in $\eps_I\in [-1,1]$.
\end{thm}

\begin{proof}
Just by reversing the argument of Theorem \ref{tuda}.

\end{proof}

\bigskip

\noindent{\bf Remark.} Notice that the Bellman function $B$ defined above  satisfies by definition $B(F, f, \la)= B(F, -f, \la)$. Therefore,  Lemma \ref{obstLemma} claims  in particular that  $B(F, f, \la)=1$ if $\la<0$ (and  we saw that it also satisfies  \eqref{mi1}, \eqref{mi2}).

\medskip

Here is the Bellman function for unweighted weak type inequality for martingale transform, see \cite{RVaVo}.
\begin{thm}
\begin{equation}
\label{full}
B(F,f,\la)= \begin{cases} 1, \,\, \text{if} \,\, \la\le F\,, \\
1-\frac{(\la-F)^2}{\la^2-f^2}\,\,\text{if}\,\, \la > F\,.\end{cases}
\end{equation}
\end{thm}

\medskip

In \cite{RVaVo} this formula was found by the use of Monge--Amp\`ere equation. As always in stochastic optimal control related problems (and this is one of such, see the explanation in \cite{NTVot}) one needs to prove that the solution of Bellman equation is actually the Bellman function. This is called ``verification theorem", and it is proved in \cite{RVaVo} as well.

\section{Weighted estimate. $A_1$ case}
\label{a1}

We keep the notations--almost. Now $w$ will be not an arbitrary weight but a dyadic $A_1$ weight. Meaning that
$$
\forall I \in D\,\, \langle w\rangle_I \le Q \inf_I w\,.
$$
The best $Q$ is called $[w]_{A_1}$. Now
$$
F=\langle |f|w\rangle_I , f=\langle f \rangle_I , \la=\la, w=\langle w\rangle_I, m=\inf_I w \,.
$$

We are in the domain
\begin{equation}
\label{O5}
\Omega:=\{(F, w, m, f, \la): F\ge |f|\,m,\,\,\, m\le w\le Q\,m\}\,.
\end{equation}
Introduce
\begin{equation}
\label{B5}
\B(F,w,m,f, \la):= \sup\,\frac1{|I|}w\{x\in I: \sum_{J\subseteq I, J\in D} \eps_J(\f, h_J) h_J(x)>\la\}\,,
\end{equation}
where the $sup$ is taken over all $\eps_J, |\eps_J|\le 1, J\in D,\, J\subseteq I$, and over all $f\in L^1(I,wdx)$ such that $F:=\langle |f|\,w\rangle_{I}\,,\,f:=\langle f\rangle_{I}
$, $ w=\langle w\rangle_I , m\le \inf_I w$,  and $w$ are dyadic $A_1$ weights, such that $
\forall I \in D\,\, \langle w\rangle_I \le Q \inf_I w$, and $Q$ being the best such constant. In other words $Q:=[w]_{A_1}^{dyadic}$.
Recall that $ h_I$ are normalized in $L^2(\R)$ Haar function of the  cube (interval) $I$, and $|\cdot |$ denote Lebesgue measure.

\subsection{Homogeneity}
\label{hom}

By definition, it is clear that
$$
s\B(F/s,w/s,m/s,f, \la)= \B(F,w,m,f, \la)\,,
$$
$$
\B(tF,w,m,tf, t\la)=\B(F,w,m,f, \la)\,.
$$
Choosing $s=m$ and $t=\la^{-1}$, we can see that
\begin{equation}
\label{Bn}
\B(F,w,m,f, \la) = m B(\frac{F}{m\la}, \frac{w}{m}, \frac{f}{\la})
\end{equation}
for a certain function $B$. Introducing new variables $\al= \frac{F}{m\la}, \beta=\frac{w}{m} ,\gamma=\frac{f}{\la}$ we write that $B$ is defined in
\begin{equation}
\label{G3}
G:=\{(\al, \beta, \gamma): |\gamma|\le \al, 1\le \beta\le Q\}\,.
\end{equation}

\subsection{The main inequality}

\begin{thm}
\label{tudaweight}
Let $P, P_+,P_-\in \Om, P=(F,w,\min(m_+, m_-), f,\la)$, $P_+=(F+\al, w+\gamma, m_+, f+\beta, \la +\beta)$, $P_-=(F-\al, w-\gamma, m_-, f-\beta, \la -\beta)$. Then
\begin{equation}
\label{mi11}
\B(P)-\frac12(\B(P_+)+\B(P_-))\ge 0\,.
\end{equation}
At the same time, if $P, P_+,P_-\in \Om, P=(F,w, \min(m_+, m_-),f, \la)$, $P_+=(F+\al,   w+\gamma, m_+,f+\beta, \la -\beta)$, $P_-=(F-\al, w+\gamma, m_+,f-\beta,  \la +\beta)$. Then
\begin{equation}
\label{mi21}
\B(P)-\frac12(\B(P_+)+\B(P_-))\ge 0\,.
\end{equation}
In particular, with fixed $m$, and with all points being inside $\Omega$ we get
$$
\B(F,  w, m, f, \la) -\frac14 (\B(F-dF, w-dw, m, f-d\la, \la-d\la) +\B(F-dF, w-dw, m, f+d\la, \la-d\la) +
$$
\begin{equation}
\label{4conc}
\B(F+dF, w+dw, m, f-d\la, \la+d\la) +\B(F+dF, w+dw, m, f+d\la, \la+d\la) )\ge 0\,.
\end{equation}
\end{thm}

\bigskip

\noindent{\bf Remark.}1) Differential notations $dF, dw, d\la$ just mean small numbers.
2) In \eqref{4conc} we loose a bit of information (in comparison to \eqref{mi11},\eqref{mi21}), but this is exactly \eqref{4conc} that we are going to use in the future.

\bigskip

\begin{proof}
Fix $P, P_+,P_-\in \Om$.
Let $\varphi_+, \varphi_-$, $w_+, w_-$ be functions and weights giving the supremum in $B(P_+), B(P_-)$ respectively up to a small number $\eta>0$.
Using the fact that $\B$ does not depend on $I$, we think that $\varphi_+, w_+$ is on $I_+$ and $\varphi_-, w_-$ is on $I_-$. Consider
$$
\f(x):=\begin{cases} \varphi_+(x)\,,\, x\in I_+\\ \varphi_-(x)\,,\, x\in I_-\end{cases}
$$
$$
\omega(x):=\begin{cases} w_+(x)\,,\, x\in I_+\\ w_-(x)\,,\, x\in I_-\end{cases}
$$
Notice that then
\begin{equation}
\label{beta11}
(\f, h_I)\cdot\frac1{\sqrt{|I|}} = \beta\,.
\end{equation}
Then it is easy to see that
\begin{equation}
\label{av11}
\langle | \f| \omega\rangle_I = F=P_1, \,\,\,\langle \f \rangle_I =f=P_4\,.
\end{equation}
Notice that for $x\in I_+$ using \eqref{beta11}, we get if $\eps_I =-1$
$$
\frac1{|I|}w_+\{x\in I_+: \sum_{J\subseteq I_+, J\in D} \eps_J(\f, h_J) h_J(x)>\la\}= \frac1{|I|}w_+\{x\in I_+: \sum_{J\subseteq I_+, J\in D} \eps_J(\f, h_J) h_J(x)>\la+\beta\}
$$
$$
=\frac1{2|I_+|}w_+\{x\in I_+: \sum_{J\subseteq I_+, J\in D} \eps_J(\varphi_+, h_J) h_J(x)>P_{+,3}\}\ge \frac12 B(P_+)-\eta\,.
$$
Similarly, for $x\in I_-$ using \eqref{beta11}, we get if $\eps_I =-1$
$$
\frac1{|I|}w_-\{x\in I_-: \sum_{J\subseteq I, J\in D} \eps_J(\f, h_J) h_J(x)>\la\}= \frac1{|I|}w_-\{x\in I_-: \sum_{J\subseteq I_-, J\in D} \eps_J(\f, h_J) h_J(x)>\la-\beta\}
$$
$$
=\frac1{2|I_-|}w_-\{x\in I_-: \sum_{J\subseteq I_-, J\in D} \eps_J(\varphi_-, h_J) h_J(x)>P_{-,3}\}\ge \frac12 B(P_-)-\eta\,.
$$
Combining the two left hand sides we obtain for $\eps_I=-1$
$$
\frac1{|I|}\om\{x\in I_+: \sum_{J\subseteq I, J\in D} \eps_J(\f, h_J) h_J(x)>\la\}\ge \frac 12 (B(P_+)+B(P_-))-2\eta\,.
$$
Let us use now the simple information \eqref{av11}: if we take the supremum in the left hand side over all functions $\varphi$, such that $\langle |\varphi | \,w\rangle_I =F, \langle \varphi \rangle_I =f , \langle\om \rangle=w$, and weights $\omega$: $\langle \omega\rangle =w$, in dyadic $A_1$ with $A_1$-norm at most $Q$, and supremum over all $\eps_J=\pm 1$ (only $\eps_I=-1$ stays fixed), we get a quantity smaller or equal than the one, where we have the supremum over all functions $\varphi$, such that $\langle |\varphi |\, \omega\rangle =F, \langle \varphi \rangle_I =f, \langle \om\rangle =w $, and weights $\omega$: $\langle \omega\rangle =w$, in dyadic $A_1$ with $A_1$-norm at most $Q$, and an unrestricted  supremum over all $\eps_J=\pm 1$ including $\eps_I=\pm 1$. The latter quantity is of course $\B(F, w, m,f, \la)$. So we proved \eqref{mi1}.

To prove \eqref{mi2} we repeat verbatim the same reasoning, only keeping now $\eps_I=1$. We are done.

\end{proof}

\medskip

\noindent{\bf Remark.} This theorem is a sort of ``fancy" concavity property, the attentive reader would see that \eqref{mi11}, \eqref{mi21} represent bi-concavity not unlike demonstrated by the celebrated  Burkholder's function. We will use the consequence of bi-concavity encompassed by \eqref{4conc}. There is still another concavity if we allow to have $|\eps_J|\le 1$.

\bigskip

\begin{thm}
\label{tudaVYP}
In the definition of $\B$ we allow now to take supremum over all $|\eps_j|\le 1$. Let $P, P_+,P_-\in \Om, P=(F,w,m, f,\la)$, $P_+=(F+\al, w+\gamma, m, f+\beta, \la )$, $P_-=(F-\al, w-\gamma, m, f-\beta, \la )$. Then
\begin{equation}
\label{3conc}
\B(P)-\frac12(\B(P_+)+\B(P_-))\ge 0\,.
\end{equation}
\end{thm}
\begin{proof}
We repeat the proof of \eqref{mi11} but with $\eps_I=0$.

\end{proof}

\begin{thm}
\label{tudaDM}
For fixed $F, w, f,\la$ function $\B$ is decreasing in $m$.
\end{thm}
\begin{proof}
Let $m=\min(m_-,m_+)= m_-$. And let $m_+>m$. Then \eqref{mi11} becomes
$$
\B(F, w, m, f, \la)- \B(F, w, m_+, f, \la) \ge 0\,.
$$
This is what we want.
\end{proof}

\subsection{Differential properties of $\B$ translated to differential properties of $B$}
\label{diff}

It is convenient to introduce an auxiliary functions of $4$ and $3$ variables:
$$
\wt B(x, y,f, \la):= B(\frac{x}{\la}, y, \frac{f}{\la})\,.
$$
Of course
\begin{equation}
\label{53}
\B(F,w,m, f, \la)= m\wt B(\frac{F}{m}, \frac{w}{m}, f, \la)=m B(\frac{F}{m\la}, \frac{w}{m}, \frac{f}{\la})\,.
\end{equation}
\begin{lm}
\label{incr}
Function $B$ increases in the first and in the second variable.
\end{lm}

\begin{proof}
We know that by definition the RHS of \eqref{53} is getting bigger if $\la $ is getting smaller. So let us consider $\la_1>\la_2, \la_1=\la_2+\delta$, and variables $F, w, m, f$ fixed, and choose $\phi_1$ (and a weight $\omega$), $\langle \phi_1\rangle=f+\eps, \langle |\phi_1|\omega\rangle =F$, which almost realizes the supremum $\B(F, w, m, f+\eps, \la_1)$.  Consider $\phi_2$ such that $\phi_2= \phi_1 -h$. Function $h$ will be chosen later, however we say now that $h$ is equal to a certain constant $a$ on a small dyadic interval $\ell$ and is zero otherwise. Constant $a$ and interval $\ell$ we will chose later. But $\eps:=\langle h\rangle$  will be chosen very soon. Function $\phi_2$ competes for supremizing $\B$ at $(\langle|\phi_2|\omega\rangle, w, m, f, \la_2)$. We choose $\eps$ in such a way that
\begin{equation}
\label{epsdelta}
\frac{\langle \phi_1\rangle}{\la_1} = \frac{f+\eps}{\la_1} =\frac{f}{\la_1-\delta}=\frac{\langle \phi_2\rangle}{\la_2}\,.
\end{equation}
Let us prove that \eqref{epsdelta} implies that
\begin{equation}
\label{F1F2}
\frac{\langle |\phi_1|\omega\rangle}{\la_1}  \le \frac{\langle |\phi_2|\omega\rangle}{\la_2} \,.
\end{equation}
By \eqref{epsdelta} this is the same as

$$
\frac{\langle |\phi_2 +h|\omega\rangle}{\langle |\phi_2|\omega\rangle}  \le  \frac{\langle \phi_1\rangle}{\langle \phi_2\rangle}= \frac{\langle \phi_2\rangle+\eps}{\langle \phi_2\rangle}\,.
$$
The previous inequality becomes
$$
\frac{\langle |\phi_2 +h|\omega\rangle}{\langle |\phi_2|\omega\rangle}  \le 1+\frac{\langle h\rangle}{\langle\phi_2\rangle} \,.
$$
By triangle inequality the latter inequality would follow from the following one
$$
\langle|\phi_2|\,\omega\rangle \ge \langle \phi_2\rangle\frac{\langle |h|\,\omega\rangle}{\langle h\rangle}\,.
$$
We can think that the minimum $m$ of $\omega$  is attained on a whole tiny dyadic interval $\ell$ (we are talking about {\it almost} supremums).
Put $h$ to be a certain $a>0$ on this interval and zero otherwise. Of course we choose $a$ to have $\langle h\rangle = \eps$, where $\eps$ was chosen before. Now the previous display inequality becomes
$$
\langle |\phi_2|\,\omega\rangle \ge \langle \phi_2\rangle \cdot m\,,
$$
which is obvious.

Notice that $\B(\langle|\phi_2|\rangle, w, m, f, \la_2)$ as a supremum is larger than the $\omega$-measure of the level set $>\la_2$ of the martingale transform of $\phi_2$. But this is also the martingale transform of $\phi_1$. The $\la_1$-level set for any martingale transform of $\phi_1$ is smaller, as $\la_1>\la_2$. But recall that we already said that $\phi_1$ (and weight $\omega$) almost realizes its own supremum $\B(F, w, m, f+\eps, \la_1)= \B(\langle |\phi_1|\rangle, w, m, \langle\phi_1\rangle, \la_1) $
So
$$
\B(\langle |\phi_1|\rangle, w, m, \langle\phi_1\rangle, \la_1) \le \B(\langle |\phi_2|\rangle, w, m, \langle\phi_2\rangle, \la_2)\,.
$$
In other notations we get
$$
B (\frac{\langle |\phi_1|\rangle}{m\la_1}, \frac{w}{m}, \frac{\langle \phi_1\rangle}{\la_1}) \le B (\frac{\langle |\phi_2|\rangle}{m\la_2}, \frac{w}{m}, \frac{\langle \phi_2\rangle}{\la_2}) \,.
$$
Let us denote the argument on the LHS as $(x_1, y_1, z_1)$, and on the RHS as $(x_2, y_2, z_2)$. Notice that $y_1= y_2=:y$ trivially and $z_1=z_2=:z$ by \eqref{epsdelta}.
Notice also that $x_1< x_2$ by \eqref{F1F2}.  Moreover by choosing $\delta$ very small we can realize any $x_1<x_2$ as close to $x_2$ as we want. Then the last display inequality reads as
$$
B(x_1, y, z) \le B(x_2, y, z)\,.
$$
So we proved that function $B$ increases in the first variable.

The increase in the second variable is easy. Choose a dyadic interval $I$ on which $\inf_I \omega>m$, but $\langle \omega\rangle_I/\inf_I\omega <Q=:[\omega]_{A_1}$. For non-constant $\omega$ this is always possible, just take a small interval containing a point $x_0$, where $\omega(x_0) >m$.
Then augment $\omega$ on $I$ slightly to get $\omega_1$ with $\langle \omega_1\rangle=w+\eps$. It is easy to see that as a result we have the new weight with the $A_1$ norm at most $Q$, the same global infimum $m$ but a larger global average $\langle \omega\rangle$. The $\omega_1$ measure of the level set of the martingale transform will be bigger than $\omega$ measure of the same level set of the same martingale transform, and $w/m$ also grows to $(w+\eps)/m$. All other variables stay the same. So if the original $\omega$ (and some $\phi$) were (almost) realizing supremum, we would get
$$
B(x, y_1, z) \le B(x, y_2, z)
$$
for $y_1=w/m, y_2=(w+\eps)/m$.
\end{proof}

\begin{thm}
\label{bigform}
Function $B$ from \eqref{Bn} satisfies
\begin{equation}
\label{Bt}
t\rightarrow t^{-1} B(\al t, \beta t, \gamma)\,\,\text{is increasing for}\,\,\frac{|\gamma|}{\al}\le t \le \frac{Q}{\beta}\,.
\end{equation}
\begin{equation}
\label{BVYP}
B\,\, \text{is concave}\,.
\end{equation}
$$
B(\frac{x}{\la}, y, \frac{f}{\la}) -\frac14 \bigg[ B(\frac{x-dx}{\la-d\la}, y-dy, \frac{f-d\la}{\la-d\la}) +B(\frac{x-dx}{\la-d\la}, y-dy, \frac{f+d\la}{\la-d\la})+
$$
\begin{equation}
\label{Bform}
 B(\frac{x+dx}{\la+d\la}, y+dy, \frac{f-d\la}{\la+d\la}) +B(\frac{x+dx}{\la+d\la}, y+dy, \frac{f+d\la}{\la+d\la})\bigg]\ge 0
\,.
\end{equation}
\end{thm}

\begin{proof}
These relations follow from Theorem \ref{tudaDM}, Theorem \ref{tudaVYP}, and Theorem \ref{tudaweight} (actually from \eqref{4conc}) correspondingly.
\end{proof}

We can choose extremely small $\e_0$ and inside the domain $\Om$ we can mollify $\B$ by a convolution of it with $\e_0$-bell function $\psi$ supported in a ball of radius $\e_0/10$.

Multiplicative convolution can be viewed as the integration with $\frac{1}{\delta^5}\psi(\frac{x-x_0}{\delta})$, where $\delta=\e_0/10$. Here $x_0$ is a point inside the domain of definition $\Om$ for function $\B$.

This new function we call $\B$  again. It is exactly as the initial function $\B$, and it obviously satisfies all the same relationships, in particular it satisfies Theorems \ref{tudaweight}, \ref{tudaVYP}, \ref{tudaDM}. Only its domain of definition$\Om_{\e_0}$ is smaller (slightly) than $\Om$. The advantage however is that the new $\B$ is smooth. We build $B$ by this new $\B$. A new function $B$ defined by the new  $\B$ as in \eqref{53} will be smooth.  Actually the new $B$ should be denoted $\bb$, where superscript denotes our operation of mollification, but we drop the superscript for the sake of brevity. In fact, all these mollifications are for the sake of convenience, the new functions satisfy the old inequalities in the uniform way, independently of $\e_0$. Property \eqref{Bform} can be now rewritten  by the use of Taylor's formula:

\begin{thm}
\label{Bdiffform}
$$
-\al^2 B_{\al\al}\bigg(\frac{dx}{x}-\frac{d\la}{\la}\bigg)^2 -\beta^2 B_{\beta\beta} \Bigg(\frac{dy}{y}\bigg)^2 -(1+\gamma^2)B_{\gamma\gamma} \Bigg(\frac{d\la}{\la}\bigg)^2-
$$
$$
-2\al\beta B_{\al\beta}\bigg(\frac{dx}{x}-\frac{d\la}{\la}\bigg)\frac{dy}{y} + 2\beta\gamma B_{\beta\gamma} \frac{dy}{y}\frac{d\la}{\la} + 2\al\gamma B_{\al\gamma}\bigg(\frac{dx}{x}-\frac{d\la}{\la}\bigg)\frac{d\la}{\la}+
$$
$$
+2\al B_{\al}\bigg(\frac{dx}{x}-\frac{d\la}{\la}\bigg)\frac{d\la}{\la} -2\gamma B_{\gamma}\bigg(\frac{d\la}{\la}\bigg)^2 \ge 0\,.
$$
\end{thm}
\begin{proof}
This is just Taylor's formula applied to \eqref{Bform}.
\end{proof}

Denoting
$$
\xi=\frac{dx}{x}=\frac{dy}{y}\,,\,\,\eta=\frac{d\la}{\la}
$$
we obtain the following quadratic form inequality
\begin{thm}
\label{xieta}
$$
-\xi^2 \,[\al^2 B_{\al\al} +\beta^2 B_{\beta\beta} + 2\al\beta B_{\al\beta}] -\eta^2\, [\al^2 B_{\al\al} + (1+\gamma^2) B_{\gamma\gamma}  + 2\al\gamma B_{\al\gamma} +2\al B_{\al} +2\gamma B_{\gamma}] +
$$
$$
+2\xi\eta \,[\al^2 B_{\al\al} +\al\beta B_{\al\beta}+ \beta\gamma B_{\beta\gamma} +\al\gamma B_{\al\gamma} +\al B_{\al}] \ge 0\,.
$$
\end{thm}

Now let us combine Theorem \ref{xieta} and Theorem \ref{tudaVYP}.
In fact, Theorem \ref{tudaVYP} implies
$$
-2\al\gamma B_{\al\gamma} \eta^2 \le -\al^2\gamma B_{\al\al}\eta^2 -\gamma B_{\gamma\gamma} \eta^2\,.
$$
We plug it into the second term above. Also Theorem \ref{tudaVYP} implies
$$
2\al\gamma B_{\al\gamma}\xi\eta   \le -\al^2\gamma B_{\al\al} \xi^2 - \gamma B_{\gamma\gamma} \eta^2\,,
$$
$$
2\beta\gamma B_{\beta\gamma}\xi\eta   \le -\beta^2\gamma B_{\beta\beta} \xi^2 - \gamma B_{\gamma\gamma} \eta^2\,,
$$
We will plug it into the third term above. Then using the notation
$$
\psi(\al, \beta,\gamma) := -\al^2 B_{\al\al}-2\al\beta B_{\al\beta} -\beta^2 B_{\beta\beta}
$$
(which is non-negative by the concavity of $B$ in its first two variables by the way)
we introduce the notations
$$
K:= \psi(\al,\beta,\gamma) + (-\al^2 B_{\al\al}-\beta^2B_{\beta\beta})\gamma\,,
$$
$$
L:=-\psi(\al,\beta,\gamma) +(\al^2B_{\al})_{\al} -\beta^2B_{\beta\beta}\,,
$$
$$
N:=-(1+3\gamma+\gamma^2) B_{\gamma\gamma} - 2\gamma B_{\gamma} -(\al^2B_{\al})_{\al} -\al^2B_{\al\al}\gamma\,.
$$
And we get that the following quadratic form is non-negative:
$$
\xi^2\, K+\xi\eta\,L +\eta^2\,N:=
$$
$$
\xi^2\,[ \psi(\al,\beta,\gamma) + (-\al^2 B_{\al\al}-\beta^2B_{\beta\beta})\gamma]+
$$
$$
\xi\eta\, [ -\psi(\al,\beta,\gamma) +(\al^2B_{\al})_{\al} -\beta^2B_{\beta\beta}]+
$$
$$
\eta^2\, [-(1+3\gamma+\gamma^2) B_{\gamma\gamma} - 2\gamma B_{\gamma} -(\al^2B_{\al})_{\al} -\al^2B_{\al\al}\gamma] \ge 0\,.
$$
Therefore,  $K$ is positive, and
\begin{equation}
\label{KLN}
N\ge \frac{L^2}{4K}\,.
\end{equation}

Now we will estimate $L$ from below, $K$ from above  and as a result we will obtain the estimate of $N$ from below, which will bring us our proof.

But first we need some a priori estimates, and for that we will need to mollify $B=\bb$ in variables $\al, \beta$. Again we make a multiplicative convolution with a bell-type function. Let us explain why we need it.
Let
$$
\hat{Q}:=\sup_G B/\al\,.
$$
We want to prove that
\begin{equation}
\label{hatQ}
\hat{Q}/Q\rightarrow \infty\,.
\end{equation}
First we need to notice that
\begin{equation}
\label{inavpsi}
\int_{1/2}^1 \psi(\al t,\beta t, \gamma)\,dt \le C\,(\hat{Q}\gamma +\frac{\hat{Q}}{Q}\al), \,\,
\psi(\al, \beta,\gamma) := -\al^2 B_{\al\al}-2\al\beta B_{\al\beta} -\beta^2 B_{\beta\beta}\,.
\end{equation}

In fact, consider $\beta\in [Q/4, Q/2]$, $b(t):= B(\al t, \beta t, \gamma)$ on the interval $\frac{|\gamma|}{\al}=:t_0 \le t\le 1$. Let $\ell(t)=b(1)t\le \hat{Q}t\alpha$.
We saw that $b(t)/t$ is increasing and $b$ is concave, and $b$ is under $\ell$, and so by elementary picture of concave function having property $b(\cdot)/\cdot$ increasing and $b(\cdot)$ concave on the interval $[t_0',1]$ we get that the maximum  of $\ell(\cdot)- b(\cdot)$ is attained on the left end-point. The left end-point $t_0'$ is the maximum of $t_0=|\gamma |/\al$ and $1/\beta$ which is $c/Q$. Therefore,
$$
\ell(t)-b(t)|(t=(\max(\frac{\gamma}{\al},\frac{c}{Q})) \le\ell(\max( \frac{\gamma}{\al},\frac{c}{Q}))\le C\hat{Q}\al\max(\frac{\gamma}{\al}, \frac{1}{Q})\le \hat{Q}\gamma +\frac{\hat{Q}}{Q}\al\,,
$$
and the above value is maximum of $g(t):=\ell(t)-b(t)$ on $[t_0',1]$.  By the same property that $b(t)/t$ is increasing we get that
$$
g'(1)=\ell'(1)-b'(1)= b(1) - b'(1)\le 0\,.
$$
Combining this with Taylor's formula on $[t_0,1]$ we get for $g:=\ell-b$ (g is convex of course):
\begin{equation}
\label{d2g}
-(1-t_0)g'(1) +\int_{t_0}^1 dt \int_{t}^1 g''(s) ds = \text{positive} + \int_{t_0}^1 (s-t_0) g''(s) ds \le \sup g  \le \hat{Q}\gamma +\frac{\hat{Q}}{Q}\al\,.
\end{equation}
This implies \eqref{inavpsi} because $g''(t) = \frac1{t^2} \psi(\al t, \beta t, \gamma), \, t\in [1/2,1]$.

Consider now function $ a(t):= B(\al t, \beta, \gamma)$
We also have the same type of consideration applied to convex function $\hat{Q}\alpha-a(t)$ bringing us

\begin{equation}
\label{inaval}
\int_{1/2}^1 -\al^2 B_{\al\al} (\al t,\beta,\gamma)\,dt \le C\hat{Q}\al\,.
\end{equation}

Similarly,

\begin{equation}
\label{inavbeta}
\int_{1/2}^1 -\beta^2 B_{\beta\beta} (\al ,\beta t,\gamma)\,dt \le C\hat{Q}\al\,.
\end{equation}

We  used here that $B_{\al}\ge 0, B_{\beta}\ge 0$, which is not difficult to see.

For the future estimates we want \eqref{inavpsi}, \eqref{inaval}, \eqref{inavbeta} to hold not in average but pointwise.

To achieve the replacement of ``in-average" estimates \eqref{inavpsi}, \eqref{inaval}, \eqref{inavbeta} by their pointwise analogs let us
consider yet another mollification, now it is of $B$:
$$
B_{new}(\al, \beta, \gamma):= 2\int_{1/2}^1 B(\al t, \beta t, \gamma)\, dt.
$$
The domain of definition of $B_{new}$ is only in tiny difference with the domain of definition of $B$. In fact, the latter is $\{(\al, \beta, \gamma):\, |\gamma|\le \al, 1\le \beta\le Q\}$, and the former is just $G:=\{(\al, \beta, \gamma):\, |\gamma|\le \frac12\al, 2\le \beta\le Q\}$. 

If we replace $(\al, \beta, \gamma)$ by $(\al t, \beta t, \gamma), 1/2\le t \le 1,$ everywhere in the inequality of Theorem \ref{xieta}, and then integrate
the inequality with $2\int_{1/2}^1\dots \,dt$, we will get Theorem \ref{xieta} but for $B_{new}$.

\bigskip 

It is not difficult to see that \eqref{inavpsi}  becomes a pointwise estimate for $B_{new}$ (just differentiate the formula for $B_{new}$ in $\al, \beta, \gamma$ and multiply by $\al, \beta,\gamma$ appropriately):
\begin{equation}
\label{navpsiPT}
-\al^2 (B_{new})_{\al\al} - 2\al\beta (B_{new})_{\al\beta} - \beta^2 (B_{new})_{\beta\beta} \le C(\hat Q \gamma+ \frac{\hat Q}{Q}\al).
\end{equation}
This pointwise estimate automatically imply new ``average" estimate:
$$
2\int_{1/2}^1 \big(-\al^2s^2 (B_{new})_{\al\al}(\al s, \beta, \gamma) - 2\al s\beta (B_{new})_{\al\beta}(., \beta,.) - \beta^2 (B_{new})_{\beta\beta}\big) \le C(\hat Q \gamma+ \frac{\hat Q}{Q}\al).
$$
This means exactly that the function
$$
\tilde B:=(B_{new})_{new} := 2\int_{1/2}^1 B(\al s, \beta, \gamma)\, ds
$$ 
still satisfies \eqref{navpsiPT}.  It also clearly satisfies the inequality of  Theorem \ref{xieta} because (as we noticed above) $B_{new}$ satisfies this inequality. To see this fact just replace all $\al$'s in  the inequality of  Theorem \ref{xieta}  applied to $B_{new}$ by $\al s$ and integrate $2\int_{1/2}^1\dots\,ds$. 

Now let us see that $\tilde B=(B_{new})_{new} $ also  satisfies a pointwise analog of \eqref{inaval}, namely, that
\begin{equation}
\label{navalPT}
-\al^2 \tilde B_{\al\al}(\al, \beta, \gamma) \le C\hat Q\al\,.
\end{equation}
To show \eqref{navalPT} we just repeat what has been done above. Let $\tilde g(t):= \hat Q\al -  B_{new}(\al t, \beta, \gamma)$. Then we have: 1) $0\le \tilde g \le \hat Q\al$ on $[t_0, 1]$, 2) $\tilde g'(1) \le 0$ (we saw that $B$, and hence $B_{new}$, are increasing in the first argument), 3) $\tilde g$ is convex. Then we saw in \eqref{d2g} that
$$
\int_{1/2}^1 s^2\,\tilde g''(s) \,ds \le\int_{1/2}^1 \tilde g''(s) \,ds \le C\hat Q\al.
$$
But this is exactly \eqref{navalPT}.  

So far we constructed a function $\tilde B= (B_{new})_{new}$ that satisfies pointwise inequalities \eqref{navpsiPT}, \eqref{navalPT} and the inequality of Theorem \ref{xieta}. We are left to see that by introducing
$$
\hat B:= 2\int_{1/2}^1 \tilde B(\al , \beta s, \gamma)\, ds
$$
we  keep \eqref{navpsiPT}, \eqref{navalPT} and the inequality of Theorem \ref{xieta} valid and also ensure
\begin{equation}
\label{navbetaPT}
-\beta^2 \hat B_{\beta\beta}(\al, \beta, \gamma) \le C\hat Q\al\,.
\end{equation}

W already just saw that \eqref{navpsiPT}, \eqref{navalPT} and the inequality of Theorem \ref{xieta} are valid for $\hat B$ just by averaging the same inequalities for  $\tilde B$. We can see that \eqref{navbetaPT} holds by the repetition of what has been just done. Namely, consider  $\hat g(t):= \hat Q\al -  \tilde B(\al, \beta t, \gamma)$. Then we have: 1) $0\le \hat g \le \hat Q\al$ on $[t_0, 1]$, 2) $\hat g'(1) \le 0$ (we saw that $B$, and hence $B_{new}$, $\tilde B$ are increasing in the first argument), 3) $\hat g$ is convex. Using \eqref{d2g} again in exactly the same manner as we did with proving \eqref{navalPT} we get 
$$
\int_{1/2}^1 s^2\,\hat g''(s) \,ds \le\int_{1/2}^1 \hat g''(s) \,ds \le C\hat Q\al.
$$
But this is exactly \eqref{navbetaPT}.  

\bigskip

We drop ``hat", and from now on $\hat B$ is just denoted by $B$. We can summarize its properties as follows.

\begin{equation}
\label{psi}
0\le \psi(\al ,\beta , \gamma) \le C(\hat{Q}\gamma +\frac{\hat{Q}}{Q}\al)\,.
\end{equation}
\begin{equation}
\label{al}
0\le - \al^2 B_{\al\al} (\al ,\beta,\gamma)\le C\hat{Q}\al\,.
\end{equation}
\begin{equation}
\label{beta}
0\le - \beta^2 B_{\beta\beta} (\al ,\beta ,\gamma) \le C\hat{Q}\al\,.
\end{equation}


Recall that (now with this mollified $B$):
$$
\xi^2\, K+\xi\eta\,L +\eta^2\,N:=
$$
$$
\xi^2\,[ \psi(\al,\beta,\gamma) + (-\al^2 B_{\al\al}-\beta^2B_{\beta\beta})\gamma]
$$
$$
\xi\eta\, [ -\psi(\al,\beta,\gamma) +(\al^2B_{\al})_{\al} -\beta^2B_{\beta\beta}]
$$
$$
\eta^2\, [-(1+3\gamma+\gamma^2) B_{\gamma\gamma} - 2\gamma B_{\gamma} -(\al^2B_{\al})_{\al} -\al^2B_{\al\al}\gamma] \ge 0\,.
$$

We will choose soon appropriate $\al_0, \al_1\le \frac1{100}\al_0$ and $\gamma\le \tau\al_0$ with some small  $\tau$. Let us introduce
$$
k:=\int_{\al_1}^{\al_0} K\, d\al = \int_{\al_1}^{\al_0}[ \psi(\al,\beta,\gamma) + (-\al^2 B_{\al\al}-\beta^2B_{\beta\beta})\gamma]\,d\al\,,
$$
$$
n:=\int_{\al_1}^{\al_0}  N\,d\al = \int_{\al_1}^{\al_0} [-(1+3\gamma+\gamma^2) B_{\gamma\gamma} - 2\gamma B_{\gamma} -(\al^2B_{\al})_{\al} -\al^2B_{\al\al}\gamma]\,d\al\,,
$$
$$
\ell:=\int_{\al_1}^{\al_0} [ -\psi(\al,\beta,\gamma) +(\al^2B_{\al})_{\al} -\beta^2B_{\beta\beta}]\,d\al\,.
$$

\noindent{\bf Estimate of $k$ from above.} The integrand of $k$ is obviously positive and $\psi$ term dominates other terms (by \eqref{psi}, \eqref{al}, \eqref{beta} and the smallness of $\gamma$).
Therefore,
\begin{equation}
\label{k}
0\le k\le C_1\,  (\hat{Q}\gamma\al_0 + C\frac{\hat{Q}}{Q}\al_0^2) + C_2\,\hat{Q}\gamma\al_0^2\le C\,  (\hat{Q}\gamma\al_0 + C\frac{\hat{Q}}{Q}\al_0^2)\,,
\end{equation}
if $Q$ is very large.
We choose (we are sorry for a strange way of writing $\al_0$, why we do that will be seen in the next section)
\begin{equation}
\label{chooseal}
\al_0= c\,\bigg(\frac{Q}{\hat{Q}}\bigg)^{\rho}\,,\,\rho=1\,,\,\al_1= \frac1{100}\sqrt{ \frac{Q}{\hat{Q}}}\al_0\,.
\end{equation}
Here $c$ is a small positive constant. We also choose to have $\gamma$ running only on the following interval
\begin{equation}
\label{g}
\gamma\in [0,\gamma_0]\,,\,\, \gamma_0:= \tau \bigg(\frac{Q}{\hat{Q}}\bigg)^{\rho} \al_0\,,\,\rho=1\,,
\end{equation}
where $\tau$ is a small positive constant.

\medskip

\noindent{\bf Estimate of $\ell$ from below.}  Estimating from below we can skip the non-negative term $ -\beta^2B_{\beta\beta}$. Also
$$
\int_{\al_1}^{\al_0} -\psi(\al,\beta,\gamma)\ge -C\hat{Q}\gamma\al_0 - C\frac{\hat{Q}}{Q}\al_0^2\,.
$$
On the other hand,
$$
\int_{\al_1}^{\al_0}(\al^2B_{\al})_{\al} \,d\al \ge \al_0^2 B_{\al}(\al_0,\beta, \gamma) -\al_1^2 \hat{Q} \,,
$$
as mollification gives a pointwise estimate
\begin{equation}
\label{Bal}
 B_{\al} \le C\hat{Q}\,.
 \end{equation}

 Recall that $\beta\in [Q/4, Q/2]$. We also will prove soon the obstacle condition \eqref{againobst}, which says that

 \begin{equation}
 \label{boundary}
 B(1,\beta, \gamma)\ge \frac{\beta}{8}\,.
 \end{equation}

  If $B_{\al}(\al_0, \beta, \gamma)$ would be smaller than $Q/40$ (and then $B_{\al}(s, \beta, \gamma)\le Q/40$ for all $s\in [\al_0,1]$ by concavity of $B$ in its first variable) we would not be able to reach at least  $\frac{Q}{4\cdot 8}$.  In fact, by our choice of $\al_0$ in \eqref{chooseal} we have
  \begin{equation}
  \label{Bal0}
  B(\al_0, \beta, \gamma) \le \hat{Q}\al_0\le c \,Q\,.
  \end{equation}
  If $B_{\al}(\al_0, \beta, \gamma) \le \frac{Q}{40}$, and so this derivative $B_{\al}(s, \beta, \gamma) \le  \frac{Q}{40}$ on $s\in [\al_0,1]$ (concavity), we cannot reach $Q/(4\cdot 8)$ for $s=1$ if we start with value of $B$ in \eqref{Bal0} at $s=\al_0$. But the fact that we cannot reach $Q/(4\cdot 8)$  contradicts to \eqref{boundary}.
  Therefore,
  \begin{equation}
  \label{Bal0b}
  B_{\al}(\al_0, \beta, \gamma) \ge \frac{Q}{40}\,,
  \end{equation}
  and
  \begin{equation}
  \label{elldominate}
 \ell\ge  \frac{\al_0^2}{40} Q -\al_1^2 \hat{Q} -C\,\hat{Q} \gamma\al_0-C\frac{\hat{Q}}{Q}\al_0^2\,.
 \end{equation}

As $\al_1=\frac1{100}\al_0 \sqrt{\frac{Q}{\hat{Q}}}$ (see \eqref{chooseal}), the second term is dominated by the first; the third term is dominated by the first because of the choice of $\gamma_0$ in \eqref{g}, the fourth term is dominated by the first one because $Q^2>>\hat{Q}$, see \cite {P} for a much better estimate.


\bigskip

Finally,
\begin{equation}
\label{ellbelow}
\ell\ge \frac{\al_0^2}{80} Q\ge c\,\al_0^2\,Q\,.
\end{equation}
And $k$ is
$$
0\le k  \le C\,  (\hat{Q}\gamma\al_0 + C\frac{\hat{Q}}{Q}\al_0^2)=\al_0 \hat{Q} \,(\gamma +\frac1Q \al_0)\,.
$$
We got
\begin{equation}
\label{nbelow}
n \ge \frac{\ell^2}{4k}\ge c\,\frac{\al_0^4 Q^2}{\al_0 \hat{Q} \,(\gamma +\frac1Q \al_0)}\,.
\end{equation}

\medskip

\noindent{\bf Estimate of $n$ from above.}  By \eqref{Bal0b}, \eqref{Bal} and \eqref{al} we get

$$
\int_{\al_1}^{\al_0} -(\al^2 B_{\al})_{\al}\, d \al -\gamma\,\int_{\al_1}^{\al_0} \al^2 B_{\al\al}\, d\al \le - c Q\al_0^2 + C\hat{Q}\al_1^2+ c\hat{Q}\al_0^2\gamma \le 0\,.
$$
Negativity is by the choice of $\al_1$ in \eqref{chooseal} and by the fact that
\begin{equation}
\label{glesq}
\gamma \le c\, \sqrt{\frac{Q}{\hat{Q}}}\,,
\end{equation}
which is much overdone in \eqref{g}.

Therefore, we get, combining with \eqref{nbelow} (here $\eta>$ is an absolute constant and  it is at least the maximum of all our $3\gamma+\gamma^2$)
$$
c\,\frac{\al_0^3 Q^2}{ \hat{Q} \,(\gamma +\frac1Q \al_0)} \le  n\le -(1+\eta)\int_{\al_1}^{\al_0} (e^{\frac1{1+\eta}\gamma^2} B_{\gamma})_{\gamma}\,d\al\,,
$$
or
\begin{equation}
\label{Bgamma}
\int_{\al_1}^{\al_0} (-e^{\frac1{1+\eta}\gamma^2} B_{\gamma})_{\gamma}\,d\al \ge   C\,\frac{\al_0^3 Q^3}{ \hat{Q} \,(Q\gamma +\al_0)} \,.
\end{equation}

Function $B$ is smooth, concave in $\gamma$ and symmetric in $\gamma$ (the latter is by definition). In particular $B_\gamma(\alpha, \beta, 0) =0$. So after integrating in $\gamma$ on $[0, \gamma], \gamma<\gamma_0$ we get
\begin{equation}
\label{BgammaInt1}
\int_{\al_1}^{\al_0} ( -B_{\gamma})\,d\al \ge  C\,\al_0^3 \frac{Q^2}{\hat{Q}}[\log (\al_0+ Q\gamma) -\log \al_0]=  C\,\al_0^3 \frac{Q^2}{\hat{Q}}\,\log (1+ \frac{Q}{\al_0}\gamma) \,.
\end{equation}

Integrate again in $\gamma$ on $[0, \gamma_0]$. We get the integral over $[\al_1,\al_0]$ of the oscillation of $B$, which is
$$
\int_{\al_1}^{\al_0} [B(\al,\beta, 0) - B(\al,\beta, \gamma_0)]\,d\al \ge C\, \al_0^3 \frac{Q^2}{\hat{Q} } \cdot \frac{\al_0}{Q} (1+ Q\frac{\gamma_0}{\al_0})\log (1+ Q\frac{\gamma_0}{\al_0})\,.
$$
But this oscillation is smaller than $C\hat{Q}\al_0^2$. We get the inequality
\begin{equation}
\label{QhatQ}
C\, \al_0^4 \frac{Q}{\hat{Q} } \,(1+ Q\frac{\gamma_0}{\al_0})\log (1+ Q\frac{\gamma_0}{\al_0}) \le \al_0^2\hat{Q}\,.
\end{equation}

\bigskip

Notice that $\al_0$, $\gamma_0$, $\gamma_0/\al_0$ are all powers of $\frac{Q}{\hat{Q}}$, which we expect to be a sort of $\frac1{(\log Q)^{p}}$.

Then we get the estimate in terms of {\it powers} of $\frac{Q}{\hat{Q}}$:
\begin{equation}
\label{QhatQ1}
 C\, \al_0^2 \frac{Q^2}{\hat{Q}^2 } \frac{\gamma_0}{\al_0}\log (1+ Q\frac{\gamma_0}{\al_0}) \le 1\,.
\end{equation}

\bigskip

Let us count the powers of $\frac{Q}{\hat{Q}}$: $\al_0^2$ brings power $2$---by \eqref{chooseal}, $\frac{\gamma_0}{\al_0}$ brings power $1$ by \eqref{g}, so totally we have
$\frac1{(\log Q)^{5p}}\log \frac{Q}{(\log Q)^{\dots}}$ in the left hand side.

\bigskip

We can see that if $\hat{Q}\le Q\log^{p} Q$ with $p <\frac1{5}$, then \eqref{QhatQ1} leads to a contradiction. So we proved
\begin{thm}
\label{onethird}
The weighted weak norm of the martingale transform  for weights $w\in A_1^{dyadic}$ can reach $c\, [w]_{A_1}\log^{p} [w]_{A_1}$ for any positive $p<1/5$.
\end{thm}

\subsection{ A small improvement: from $1/5$ to $2/7$}
\label{5}

Suppose that we are allowed to transform the martingale not just  by $\e_J=\pm 1$ but by any $|\e_j|\le 1$ (it is not clear whether this is the same for weak norm estimate, probably yes). The change will give us that $d\la|\le |df|$, and this will mean in articular, that we automatically have that function $B(\al, \beta, \gamma)$ is concave in $\gamma$. We observed that it is symmetric in $\gamma$. Together this gives us that
\begin{equation}
\label{maxg}
B(\al, \beta, \gamma)\le B(\al, \beta, 0),\, |\gamma|\le \al;\,\,B(\al, \beta, \gamma) \ge B(\al, \beta, 0)/2,\,  |\gamma|\le \al/2\,.
\end{equation}

Now to improve the constant $1/5$ we consider $Q^+:= \sqrt{Q\hat Q}$. We put
$$
a_0:= c_1 \frac{Q}{Q^+}>> \al_0.
$$
Two cases appear:

\noindent Case 1.  $B(a_0, \beta, 0) \le Q^+ a_0$. Then we replace $\al_0$ by $a_0$ in \eqref{chooseal}, we replace $\gamma_0$ by $\tilde\gamma_0=a_0 \big(\frac{Q}{Q^+}\big)^{\rho}$ in \eqref{g}, and we got \eqref{Bal0} with $B(a_0, \beta, \gamma) \le cQ$ (we use \eqref{maxg} here). And a result we have (exactly by the same reasoning as above)
$B_\al(a_0, \beta, \gamma) \ge \frac{Q}{40}$. This the same as \eqref{Bal0b} but with $a_0$ instead $\al_0$.  Now the main bookkeeping inequality
\eqref{QhatQ1} with $a_0$ replacing of $\al_0$, $\tilde\gamma_0$ replacing $\gamma_0$, gives us a new $p= 2/7$.

\bigskip

\noindent Case 2. $B(a_0, \beta, 0) \ge Q^+ a_0$. Then $B(a_0, \beta, 0) \ge c_1Q$.  And by \eqref{maxg} $B(a_0, \beta, \gamma) \ge c_1'Q$ if $|\gamma|\le a_0/2$.
But we saw that $B(\al_0, \beta, \gamma) \le cQ$. Then between $\al_0$ and $a_0$ there is a point $\tilde\al_0$ such that 
$B_\al(\tilde\al_0, \beta, \gamma) \ge c_2Q/(c_1Q/Q^+- cQ/\hat Q) \ge c_3 Q^+$. Then by concavity
$B_\al(\al_0, \beta, \gamma) \ge c_3 Q^+$. This is exactly \eqref{Bal0b}, but with a bigger constant in the right hand side ($Q^+$ in place of $Q$).

Therefore we can repeat verbatim the whole body of estimates after \eqref{Bal0b} up to the main bookkeeping inequality \eqref{QhatQ1}. However, in \eqref{QhatQ1}  $Q^2$ in the numerator should be replaced by $(Q^+)^2$. Calculating $p$ we are able again reduce it to $2/7$.

\subsection{Obstacle conditions for $B$.}
\label{obstacle}

Now we want to show the following obstacle condition for $B$, which we already used:
\begin{equation}
\label{againobst}
\text{if}\,\,|\gamma|<\frac14\,,\,\,\text{then}\,\,B(1, \beta, \gamma)\ge \frac{\beta}{8}\,.
\end{equation}

Let $I:=[0,1]$. Given numbers $|f|<\la/4, \frac{F}{m}=\la$ it is enough to construct functions $\varphi, \psi, w$ on $I$ such that


Put $\varphi=-a$ on $I_{--}$, $=b$ on $I_{++}$, zero otherwise. And $w=1$ on $I_{--}\cup I_{++}$, and $w=Q$ otherwise. Then
put
$$
\psi:= (\varphi, h_{I_-}) h_{I_-} - (\varphi,  h_{I_+}) h_{I_+}\,.
$$
Let $0<a<b$ and $a$ is close to $b$. Put $\la=(a+b)/4$. Then average of $\varphi$ is small with respect to $\la$ and we can prescribe it.
$F=(a+b)/4, m=1$. On the other hand, function $\psi$ (which is a martingale transform of $\varphi-\langle\varphi\rangle$) is at least $-(\varphi, h_{I_+}) h_{I_+}\ge \frac12 b\ge \la$ on $I_{+-}$, whose $w$-measure is more than $\frac13 w(I)$. So
\begin{equation}
\label{4betagamma}
B(1, \beta,\gamma)\ge \frac13 \beta\,,
\end{equation}
for all small $\gamma$ and $\beta\asymp Q$. This is what we wanted to prove.

\section{Bellman function and the estimate of weighted weak norm from above in $A_1$ case}
\label{above}

Let us denote by $N_k$ the quantity  ($w\,\text{is constant on k-th generation and}\, w\in A_1^{dyadic}$)
$$
 N_k(V):=\sup\frac{1}{|I|}w\{x\in I: \sum_{J\subset I, J\in D, |J|\ge 2^{-k}|I|} \eps_J(f, h_J) h_J>\la\}\,.
$$

Then we have practically by the definition of $N_k$  (let $V$ temporarily denotes vector $(F,f, \la, w, m)$, and $y_1:=\la +f, y_2:=\la-f$)
\begin{equation}
\label{Nk}
N_{k+1} (V) \le \sup_{V_+, V_-, V=\frac{V_++V_-}{2}, |y_{1+}-y_{1-}|=|y_{2+}-y_{2-}|}\frac{N_k(V_-)+N_k(V_+)}{2}\,.
\end{equation}

In this language we need to prove that
\begin{equation}
\label{anykey}
N_{k}(V) \le B(V)\,\,\text{for any}\,\, k\,\,\text{and any}\,\, V \in \Omega_k\,.
\end{equation}

By bi-concavity of $B$ and by \eqref{Nk} we immediately see the induction step from $k$ to $k+1$. We are left to check that
\begin{equation}
\label{N0}
N_{0} (V) \le B(V)\,.
\end{equation}
Let us check \eqref{N0}. If $ \la > \frac{F}{m}\ge |f|$ we just use $B(V)\ge 0$ because for such parameters
$$
|(f, h_{[0,1]})|\le |f|<\la
$$
and the subset of $[0,1]$, where $\epsilon_{[0,1]}(f, h_{[0,1]})h_{[0,1]}(x)$ is greater than $\la$ is empty.

On the other hand, if $ \la \le \frac{F}{m}$, what can be the largest $w$-measure of $E\subset [0,1]$ on which  $\epsilon_{[0,1]}(f, h_{[0,1]})h_{[0,1]}(x)\ge \la$?
Here is the extremal situation: $w$ is $2Q-1$ on $[0, 1/2]$, and $1$ on $[1/2,1]$. Function $\varphi$ is zero on $[0, 1/2]$, and constant  $2f$ on $[1/2,1]$. Then $F=f, m=1$ (these are data on $[0,1]$). On the other hand,
$$
\epsilon_{[0,1]}(\varphi, h_{[0,1]})h_{[0,1]}(x)=\epsilon_{[0,1]}\,f\,h_{[0,1]}(x) =\frac{F}{m} \ge \la
$$
on the whole $[0,1/2]$ if $\epsilon_{[0,1]}=\pm$ is chosen in the right way. But in this case again, $B(V)\ge 2Q-1\ge w([0,1/2])$.

Hence, $B(V)\ge N_0(V)$ is proved, and we can start the induction procedure.

It is left to find our $ \B$ to have a sharp estimate from above in $A_1$ problem.

\section{Martingales}
\label{Mart}

We will use four-adic lattice $\F$.
For a four-adic interval $I$ let
$H_I= 1$ on its right half, $H_I =-1$ on its left half, let also
$G_I=1$ on its leftest and rightest quarters and $G_I =-1$ on two middle quarters. We will call  martingale difference  the
function of the type
$$
f_n=\sum_{I\in \F, \ell(I)=4^{-n}} a_I H_I
$$
or
$$
g_n=\sum_{I\in \F, \ell(I)=4^{-n}} b_I G_I\,,
$$
where $a_I, b_I$ are numbers.
Martingale for us is any function on $I_0:=[0,1]$ of the type
$\ff=f+\sum_{n=0}^N f_n\,,$ or $ \g=g+\sum_{n=0}^N g_n\,,$ where $f,g$ are two constants.
We distinguish $H$- and $G$-martingales.

In the previous sections the following theorem was proved.
\begin{thm}
\label{mart1}
Given $Q>1$ there exist three $H$-martingales  $\FF,\ff,\WW$, $\FF\ge 0, \WW\ge 0$, and one $G$-martingale $\g = g+\sum_{n=0}^N g_n$ with large positive $g$ such that the following holds:

1) For any $I\in \F$, $\La \FF\Ra_I \ge \La |\ff| \Ra_I\min _I\WW$.

2) For any $I\in \F$, $\La \WW\Ra_I \le Q\min_I\WW$.

3) For any $I\in \F$, $a_I=b_I$, where these are martingale differences coefficients  for $\ff$ and $\g$.

4) $g\cdot\int_{x\in I_0: \g(x) \le 0}\WW\, dx \ge c\, Q \log^{p}Q\,\int_{I_0}\FF\,dx$, $p<\frac15$.
\end{thm}

\section{Controlled doubling martingales}
\label{DoubMart}

We are going to make a small modification in the proof to get the following

\begin{thm}
\label{mart2}
Given $Q>1$ there exist three $H$-martingales  $\FF,\ff,\WW$, $\FF\ge 0, \WW\ge 0$, and one $G$-martingale $\g=\sum_{n=0}^N g_n$  such that the following holds:

1) For any $I\in \F$, $\La \FF\Ra_I \ge \La |\ff| \Ra_I\min _I\WW$.

2) For any $I\in \F$, $\La \WW\Ra_I \le Q\min_I\WW$.

3) For $I\in \F$, $b_I=-a_I$, where these are martingale differences coefficients  for $\ff$ and $\g$.

4)For large positive number $g$,  $g\cdot\int_{x\in I_0: \g(x) \ge g}\WW\, dx \ge c\, Q \log^{p}Q\,\int_{I_0}\FF\,dx$, $p<\frac15$.

5) For any two four-adic neighbors (neighbors in the tree)  $I\in \F$ and  $\hat{I}$, $\La \WW\Ra_{\hat{I}}\le 4\La\WW\Ra_I$.
\end{thm}

In other words, we can always control the four-adic doubling property of $\WW$.

\section{Remodeling by proliferation. The amplification of martingale differences}
\label{remod}

Now we are going to repeat the procedure from \cite{NV}. We say that $I_0$ supervises itself. Take a very large $n_1$, consider the division of of $I_0$ to $4^{n_1}$ small equal intervals and let the leftest quarter of the supervisor ($I_0$ it is) supervises the first, fifth, etc small subdivision interval of the supervisee (which is still $I_0$ for now, so these are intervals of our just done subdivision). Let the second quarter supervises the second, the sixth, etc; the third quarter supervises the third, the seventh, etc, and the fourth quarter supervises the fourth, the eighth, etc.

Now we have new pairs of (supervisor, supervisee). Subdivide each supervisor to its $4$ sons and its supervisee to $4^{n_2}$ sons, where $n_2>>n_1$. Repeat supervisor/supervisee allocation procedure
as before, Continue with the new pairs of (supervisor, supervisee). Repeat $N$ times.

Now, as in \cite{NV}, we are going to ``remodel" martingales $\WW, \FF, \ff, \g$ to new functions with basically the same distributions.

We first ``square sine" and ``square cosine" function for any supervisee interval $I$. Let
$$
sqsin_{I_0} (x) := H_{I_0}(4^{n_1} x)\,,\,\, sqcos_{I_0} (x)= G_{4^{n_1}I_0}( x)\,.
$$
Next supervisors will be the quarters of $I_0$. Take on of such quarter, say, $I$, and put
$$
sqsin_{I} (x) := H_{I}(4^{n_2} x)\,,\,\, sqcos_{I} (x)= G_{4^{n_2}I_0}( x)\,.
$$
We continue doing that for the next generation of supervisors.  Let $I, J$ be a supervisor/supervisee pair.
Now let $\ell_{IJ}$ be a natural linear map $J\rightarrow I$. We put $sqs_J:= sqsin_I \circ \ell_{IJ}$, $sqc_J= sqcos_I\circ \ell_{IJ}$.

Now we basically want to put
$$
W:= w+ \sum_{n=1}^N\sum_{I\in \F, \ell(I)=4^{-n}}\sum_{J\, supervised\, by \, I} c_I sqs_J\,,
$$
where $c_I$ are coefficients of $\WW$.
$$
\Phi:= F+ \sum_{n=1}^N\sum_{I\in \F, \ell(I)=4^{-n}}\sum_{J\, supervised\, by \, I} d_I sqs_J\,,
$$
where $d_I$ are coefficients of $\FF$.
$$
\phi:= f+ \sum_{n=1}^N\sum_{I\in \F, \ell(I)=4^{-n}}\sum_{J\, supervised\, by \, I} a_I sqs_J\,,
$$
where $a_I$ are coefficients of $\ff$.
$$
\rho:=  \sum_{n=1}^N\sum_{I\in \F, \ell(I)=4^{-n}}\sum_{J\, supervised\, by \, I} b_I sqc_J\,,
$$
where $b_I=-a_I$ are coefficients of $\g$.
Notice that the last formula has square cosines $sqc$, and this will be important.

We do exactly that, but to ensure the doubling property of $W$ we just for every pair $(I,J)$ (supervisor/supervisee) replace $sqs_J$ by basically the same function, but such that its first $4$ steps on the left are replaced by $0$ and its first $4$ steps on the right are replaced by zero. Call it $sqsm_J$. So
$$
W:= w+ \sum_{n=1}^N\sum_{I\in \F, \ell(I)=4^{-n}}\sum_{J\, supervised\, by \, I} c_I sqsm_J\,,
$$
where $c_I$ are coefficients of $\WW$.

The doubling property of such a $W$ has been checked in \cite{NV}.
We notice that if $1<<n_1<<n_2<<...<<n_N$ then the distribution functions of these new function are basically the same that for their model martingales. So we can repeat Theorem \ref{mart2}.
Let us consider the periodic extension of $W, \Phi, \phi, \rho$ to the whole line (or we could consider everything just on the unit circle identifying it with $[0,1)$).

\begin{thm}
\label{mart3}
Given $Q>1$ then the above functions $W, \Phi, \phi, \rho$  are such that the following holds:

1) For any $J\in \F$, $\La \Phi\Ra_J\ge \La |\phi| \Ra_J\min _J W$.

2) For any $J\in \F$, $\La W\Ra_J \le Q\min_JW$.

3) For any $J\in \F$, $b_J=-a_J$,  these are martingale differences coefficients  for $\phi$ and $\rho$.

4) For a large positive number $g$, $g\cdot\int_{x\in I_0: \rho(x) \ge g}W\, dx \ge c_0\, Q \log^{p}Q\,\int_{I_0}\Phi\,dx$, $p<\frac15$.

5) $W$ is doubling with an absolute constant.
\end{thm}

Now what happens with the Hilbert transform $H\phi$ of $\phi$? It is immediate that if we extend periodically $scsin_{I_0}$ to the real line and do the same with $sqcos_{I_0}$ and call them $sqsin$, $sqcos$, then
\begin{equation}
\label{xi}
H(sqsin)(x)=\xi(x)\, sqcos(x)\,,
\end{equation}
where $\xi$ is a non-negative $1$-periodic function 
that looks as follows. It is logarithmically goes to $+\infty$ at $0$, at $\frac12 -$, at $\frac12 +$ and at $1$.
It has two zeros: at $\frac14$ and at $\frac34$. Continue it $1$-periodically
Let $I$ be one of the supervisors of $k$-th generation. Put 
$$
\xi_I(x):= \xi(4^{n_k}x), \, x\in I\,.
$$
Let now $I,J$ is the the pair of supervisor/supervisee.
Recall that $\ell_{IJ}$ is the linear map from $J$ to $I$ sending the left (right) end-point to the left (right) end-point.
We put
$$
\xi_J := \xi_I\circ \ell_{IJ}.
$$

It is now tempting (looking at the definition of $\phi$) to write that (recall that $b_I=-a_I$)

$$
H\phi (x)=  \sum_{n=1}^N\sum_{I\in \F, \ell(I)=4^{-n}}\sum_{J\, supervised\, by \, I} b_I \xi_J(x) sqc_J(x)\,.
$$

Unfortunately, in \eqref{xi} we have $1$-periodic  $sqsin, sqcos$ and not their localized to $I_0$'s versions. But $H$ of any bounded highly oscillating function on interval $J$ goes to zero
uniformly outside the neighborhood of the end-points of $J$. Therefore we can make up for the problem with localized to $J$ functions $sqs_J, sqc_J$ by writing
\begin{equation}
\label{Hphi}
H\phi (x)=  \sum_{n=1}^N\sum_{I\in \F, \ell(I)=4^{-n}}\sum_{J\, supervised\, by \, I} b_I \xi_J(x) sqc_J(x)+\Theta(x)\,,
\end{equation}
where $\Theta(x)$ is as close to zero as we wish on  a set of as small Lebesgue measure as we wish--by the choice of largeness of $n_1<<n_2<<\dots$.

We can think that in all our constructions all sums are finite. In particular, the coefficients $a_I, b_I=-a_I$ of $\ff, \g$ can be thought to be zero after a while. So let $m_0$ be the last generation where we have these coefficients non-zero. Then the set
\begin{equation}
\label{distrFom}
\omega:=\{x\in [0,1]:\, \g(x) \ge g\}
\end{equation}
consists of the collection of the whole intervals of the next generation $m_0+1$, that is consists of the certain sons of certain collection of $4$-adic intervals of generation $m_0$ whom we will  call $\mathcal{\hat I}$. The set of their sons forming $\omega$ will be called $\mathcal{I}$.
Intervals $\hat I$ from $\mathcal{\hat I}$ are the last intervals that are supervisee in the remodeling construction above. Their supervised intervals will be called collection $\mathcal{\hat J}$. Let $I\in \mathcal{\hat I}$, $J\in \mathcal{\hat I}$ are supervisee/supervised pair. We do remodeling last time: divide $I$ to its sons, divide $J$ to $4^{n_{m_0}}$ equal intervals, make correspondence between the sons of $I$ and some small intervals of this subdivision of $J$. Let the son of $I$ happen to be in $\mathcal{ I}$, then we mark correspondent  small intervals  $J'$ of this subdivision of $J$ by red.
All red intervals will be called collection $\mathcal{J}$. Call it $\Omega:= \cup_{J'\in\mathcal{J}} J'$.

Now we can see that
\begin{equation}
\label{distrF}
\Omega = \cup_{J'\in\mathcal{J'}} J=\{x\in [0,1]:\, \rho(x) \ge g\}.
\end{equation}

In fact, let $I'$ be an element of $\mathcal{I}$, and $J'$ be a corresponding red interval (from $\mathcal{J}$). Fix any point $x\in I$ and consider
$\g=\sum_{I} b_I G_I$ at $x$. Consider the sequence of the terms of this sum.  Here the sum has a term  $b_{\hat I'} G_{\hat{I'}}(x)$ from father $\hat I'$ of $I'$, then a term from a grandfather, et cetera.
And for any $x\in I'$ this sequence we just described is the same.
If we consider now any $y\in J'$ and consider the sequence of terms in the sum $\rho(y)=\sum_{J:\, J \,supervised\, by\, I} b_I sqc_J(y)$ we will see that it is exactly the same sequence as for $x\in I'$. This was done by remodeling construction because each $J''$ that gives the contribution to the sum at $y$ has a supervisor $I''$ that gives {\it the same} contribution to the sum at $x$. This proves \eqref{distrF}. This proves that $\g$ and $\rho$ are distributed in the same way
(with respect to Lebesgue measure, and also with respect to pair $\WW$, $W$ correspondingly).

However, we need a subtler thing. The distribution of $\rho$ is not enough for us, we need also the distribution of 
$$
\tilde H\phi (y) :=  \sum_{n=1}^N\sum_{I\in \F, \ell(I)=4^{-n}}\sum_{J\, supervised\, by \, I} b_I \xi_J(y) sqc_J(y)\,.
$$

The problem is of course that we have all these $\xi_{J}(y)$. In fact, if $I'\in \mathcal{I}$, the for any   red interval $J'$ supervised by $I'$ (and any point $y$ in any such $J'$) we have one and the same sequence of numbers $\{b_I sqc_J(y)\}_{I'\subset I, \, J\, is\, supervised\, by \, I}$.

Call this sequence of numbers $d(x)$. It is a finite sequence $\{d_1,\dots, d_{m_0}\}$ and if $x\in I'\subset \omega$, then (see \eqref{distrFom})
$$
d_1+\dots d_{m_0} =: g_1\ge g\,.
$$
We normalize by $\theta_i:= d_i/g_1$.  Then in the corresponding $y\in J'$ we have the sum for $\tilde H\phi (y)$ which looks like 
$$
\sum_{i=1}^{m_0} \theta_i x_i(y),
$$
where $x_i(y)$ is a corresponding $sqc_J(y)$, for example $x_1(y)= sqc_J(y)$, where $J$ is a father of $J'$ and a supervisee of a father $I$ of $I'$. 

We have to notice that the sequence $d(x)$ does not depend on $x$, depends only on $I'\in\mathcal{I}$, and hence the sequence
$\{\theta_1,\dots, \theta_{m_0}\}$ does not depend on $y$ as long as $y\in J'$, and $J'$ is a red interval corresponding to $I'$.

But unfortunately $\xi_i(y)$ depend on $y$ very much. In different red intervals $J_1', J_2',\dots$ corresponding to  the same $I'$ the sequence $\{x_1,\dots, x_{m_0}\}$   is completely different.

Fix our $I'\in \mathcal{I}$, let $\mathcal{J}(I')$ be  the union of all red $J'$ corresponding to $I'$. On $Y:=\cup_{J'\in \mathcal{J}(I')} J'$ we introduce the probability measure as follows: choose any such $J'$ with equal probability $\PP'$, and then put a normalized Lebesgue measure on it.

Notice that the joint distribution of $\{x_1(y),\dots, x_{m_0}(y)$\}, $y\in Y$, with respect to this $\PP'$ is almost the same as the joint distribution of independent random variables $\{\xi_1,\dots \xi_{m_0}\}$ having the same distribution of our function $\xi$ on $[0,1]$. We can make closeness in joint distribution apparent by choosing very large $n_1<<n_2<<\dots$.

Consider now two cases: 1) $\sum \theta_k^2 < c_0$, 2)  $\sum \theta_k^2 \ge c_0$, where $c_0$ is a certain absolute constant.

Let $\xi = \sum\theta_k \xi_k, \zeta_k= \xi_k-\E\xi_k, \zeta= \sum\theta_k\zeta_k$. Let us think that $\int\xi=1$
Notice that then by normalization of $\theta_i$ we have  $\E(\sum\xi_k) = 1$.

Case 1).  $\PP\{ \xi<1/2\} = \PP\{|\xi-1|>1/2\} \le 4Var(\zeta) \le 4 c_0$. So if $c_0$ is happened to be $=1/8$ we get that
$$
\PP\{ \xi\ge 1/2\} \ge 1/2\,.
$$ 
Then by the closeness in joint distribution we would conclude that 
\begin{equation}
\label{sumcoef1}
\PP'\{\sum_{i=1}^{m_0} \theta_i x_i(y)>1/2 \} \ge 1/4\,.
\end{equation}

\bigskip

Case 2). In This case the sum of variations of $\theta_k\xi_k$ is sufficiently large.
Now we will use then the following lemma:
\begin{lm}\label{le83}
Let $\theta_k>0,\ k=0,\dots,m-1$. Let $\tilde\xi_k$ be $\R$-valued independent random variables with variation $\theta_k$ satisfying
\begin{equation}
\label{moments}
\E|\tilde\xi_k|^p\le Cp\,\theta_k^{p/2}, \, p=3, 4,\dots
\end{equation}
Then there exists $\delta=\delta(C,c)>0$ such that
\begin{equation}
\label{sumcoef2}
\PP\bigg\{\bigg|\sum_{k=0}^{m-1}\tilde\xi_k+a\bigg|\ge\delta\bigg(\sum_{k=0}^{m-1}\theta_k^2\bigg)^{1/2}\bigg\}\ge\delta
\ \text{ for all }\ a\in\R.
\end{equation}
\end{lm}

\bigskip

\noindent{\bf Remark.} Notice that function $\xi\in BMO$, so by John--Nirenberg inequality the requirement \eqref{moments} hold for our 
$$
\tilde\xi_k := \theta_k\xi_k.
$$
We will apply this Lemma to such $\tilde\xi_k$ and to $a=0$. Notice that our $\tilde\xi_k$ will be non-negative.

\bigskip

\begin{proof}
Denote
$$
\s=\sum_{k=0}^{m-1}\tilde\xi_k+a,\quad \zeta_k=\tilde\xi_k-\E\tilde\xi_k.
$$

Take $\la>0$ and consider
$$
|\E e^{i\la\s}|=\bigg|e^{i\la a}\prod_{k=0}^{m-1}\E e^{i\la\tilde\xi_k}\bigg|\\
=\prod_{k=0}^{m-1}|\E e^{i\la\tilde\xi_k}|=\prod_{k=0}^{m-1}|\E e^{i\la\zeta_k}|.
$$
Note now that for $\la\le \theta_k^{-1}$ (our $\la$ below will be such) we have by \eqref{moments}
$$
|\mathscr{E}e^{i\la\zeta_k}|=\bigg|1-\frac{\la^2}2Var\xi_k+O(\la^3\theta_k^3)\bigg|\le
\exp\bigg(-\frac{\la^2}2Var\xi_k+C\la^3\theta_k^3\bigg),
$$
and
$$\aligned
\prod_{k=0}^{m-1}|\mathscr{E}e^{i\la\zeta_k}|&\le
\exp\left(-c\la^2\sum\theta_k^2+C\la^3\sum\theta_k^3\right)\\
&\le\exp\bigg(-c'\left[\la\left(\sum\theta_k^2\right)^{1/2}\right]^2+
C'\left[\la\left(\sum\theta_k^2\right)^{1/2}\right]^3\bigg).
\endaligned$$
Now choose
$$
\la=\frac{c'}{2C'}\left(\sum\theta_k^2\right)^{-1/2}.
$$
Then
$$
|\mathscr{E}e^{i\la\s}|\le\exp\bigg(-\frac{(c')^3}{8(C')^2}\bigg).
$$
On the other hand, for every $\delta>0$, one has
$$
\begin{aligned}
|\E e^{i\la\s}-1|&\le\la\delta\left(\sum\theta_k^2\right)^{1/2}+
2\PP\bigg\{|\s|>\delta\left(\sum\theta_k^2\right)^{1/2}\bigg\}\\
&\le\frac{c'}{2C'}\delta+2\PP\bigg\{|\s|>\delta\left(\sum\theta_k^2\right)^{1/2}\bigg\}.
\end{aligned}
$$
Hence,
$$
\PP\bigg\{|\s|>\delta\left(\sum\theta_k^2\right)^{1/2}\bigg\}\ge
\frac12\bigg[1-\exp\bigg(-\frac{(c')^3}{8(C')^2}\bigg)-\frac{c'}{2C'}\delta\bigg]>\delta,
$$
if $\delta$ is chosen small enough.
\end{proof}

The terms of the sum $\sum_i\theta_i x_i(y)$ are almost constant functions on each red interval $J'\in \mathcal{I'}$. 
We already proved in \eqref{sumcoef1}, \eqref{sumcoef2} that probability $\PP$ 
of the sum $\sum_i\theta_k\xi_k$ is larger than certain fixed  absolute $\delta$ is at least $\delta$. 
Therefore we may think that at least $\delta/2$ portion of red intervals $J'\in \mathcal{I'}$ are such that for  the sum $\sum_i\theta_i x_i$ we have

 $$
 \min_{J'}\sum_i\theta_i x_i\ge \delta/2.
 $$
 
 Denote this collection of $J'$ by symbol $\mathcal{C}(I')$.
 
 Now the previous inequality   translates  into
 $$
 \tilde H\phi(x) \ge \frac{\delta}{4} g,
 $$
 on all $J'$ from the portions $\mathcal{C}(I')$ described above for all intervals $I'\in \mathcal{I}$.

 Now use 4) of Theorem \ref{mart3}. The estimate in 4) $\rho(x)\ge g$ holds on {\it all} red $J'$ corresponding to any $I'\in \mathcal{I}$ (see \eqref{distrF}). The $W$-measure of the union of them is large as indicated in 4), namely, $\ge \frac{c}{g} Q\log ^p Q\int |\phi(x) W(x)\, dx$.
 
  Notice that all red intervals $J'$ from $\mathcal{I'}$ have the same $W$ measure (by construction of $W$). Therefore, the $W$-measure of all these portions of red intervals described above (portions are enumerated by $I'\in \mathcal{I}$)  is at least $\delta/4$ times $c \,Q\log ^pQ \int |\phi(x)| W(x)\, dx$. So on such $W$-measure we have $ \tilde H\phi(x) \ge \frac{\delta}{4} g$.  This is exactly what we need if we take into consideration that $\Theta(x)$ in \eqref{Hphi} can be taken as small as we wish outside the set of Lebesgue measure (and then obviously of $Wdx$ measure as well) as small as we wish.

\end{document}